\documentclass[10pt,a4paper,reqno]{amsart}
\usepackage{amsthm}
\usepackage{amsmath}
\usepackage{amssymb}
\usepackage{graphicx,color}
\usepackage{subcaption}

\usepackage[font=small]{caption}

\usepackage{multirow}
\usepackage[shortlabels]{enumitem}
\usepackage{float}

\makeatletter
\theoremstyle{plain}
\newtheorem{thm}{Theorem}
  \theoremstyle{definition}
  
  \theoremstyle{remark}
  \newtheorem{rem}[thm]{Remark}
  \theoremstyle{plain}
  \newtheorem{prop}[thm]{Proposition}
  \theoremstyle{plain}
  \newtheorem{lem}[thm]{Lemma}
  \theoremstyle{plain}
  \newtheorem{cor}[thm]{Corollary}
 \theoremstyle{definition}
  
  \theoremstyle{remark}
  \newtheorem*{rem*}{Remark}

  \theoremstyle{definition}

\usepackage{amsfonts}
\usepackage{mathrsfs}

\newtheorem*{question*}{\it{QUESTION}}

\theoremstyle{plain}

\newcommand{\N}{\mathbb{N}}
\newcommand{\R}{{\mathbb{R}}}
\newcommand{\C}{{\mathbb{C}}}
\newcommand{\Z}{{\mathbb{Z}}}
\newcommand{\Q}{{\mathbb{Q}}}
\newcommand{\D}{{\mathbb{D}}}
\newcommand{\T}{{\mathbb{T}}}
\newcommand{\dd}{{\rm d}}
\newcommand{\ii}{{\rm i}}

\newcommand{\diag}{\mathop\mathrm{diag}\nolimits}
\newcommand{\spn}{\mathop\mathrm{span}\nolimits}
\newcommand{\Dom}{\mathop\mathrm{Dom}\nolimits}

\newcommand{\Ran}{\mathop\mathrm{Ran}\nolimits}
\newcommand{\Ker}{\mathop\mathrm{Ker}\nolimits}

\newcommand{\sgn}{\mathop\mathrm{sgn}\nolimits}
\renewcommand{\Re}{\mathop\mathrm{Re}\nolimits}

\catcode`,\active

\catcode`\,12

\usepackage[normalem]{ulem}
\definecolor{DarkGreen}{rgb}{0,0.5,0.1} 
\definecolor{BrinkPink}{rgb}{0.98, 0.38, 0.5} 

\newcommand\soutD{\bgroup\markoverwith
{\textcolor{DarkGreen}{\rule[.5ex]{2pt}{1pt}}}\ULon}
\newcommand{\Hm}[1]{\leavevmode{\marginpar{\tiny%
$\hbox to 0mm{\hspace*{-0.5mm}$\leftarrow$\hss}%
\vcenter{\vrule depth 0.1mm height 0.1mm width \the\marginparwidth}%
\hbox to
0mm{\hss$\rightarrow$\hspace*{-0.5mm}}$\\\relax\raggedright #1}}}

\def\section{\@startsection{section}{1}%
  \z@{.7\linespacing\@plus\linespacing}{.5\linespacing}%
  {\normalfont\scshape\bf\centering}}
\newcommand{\diff}{\mathrm{d}}

\begin{document}

\title[Discrete Schr\"odinger operators on the half-line]{Spectral enclosures and stability for non-self-adjoint discrete Schr\"odinger operators on the half-line}

\author{David Krej{\v c}i{\v r}{\' i}k}
\address[David Krej{\v c}i{\v r}{\' i}k]{
Department of Mathematics, 
Faculty of Nuclear Sciences and Physical Engineering, 
Czech Technical University in Prague,
Trojanova 13, 12000 Prague~2, Czech Republic.
	}
\email{david.krejcirik@fjfi.cvut.cz}

\author{Ari Laptev}
\address[Ari Laptev]{
Department of Mathematics,
Imperial College London,Huxley Building, 180 Queen's Gate
London SW7 2AZ, UK, and Sirius Mathematics Center,
Sirius University of Science and Technology,
1 Olympic Ave, 354340, Sochi, Russia, 
	}
\email{a.laptev@imperial.ac.uk}

\author{Franti{\v s}ek {\v S}tampach}
\address[Franti{\v s}ek {\v S}tampach]{
Department of Mathematics, 
Faculty of Nuclear Sciences and Physical Engineering, 
Czech Technical University in Prague,
Trojanova 13, 12000 Prague~2, Czech Republic.
	}
\email{stampfra@fit.cvut.cz}



\date{November 14, 2021}

\begin{abstract}
We make a spectral analysis of 
discrete Schr{\" o}dinger operators on the half-line,
subject to complex Robin-type boundary couplings
and complex-valued potentials.
First, optimal spectral enclosures are obtained for summable potentials. 
Second, general smallness conditions on the potentials guaranteeing 
a spectral stability are established.
Third, a general identity which allows to generate optimal discrete Hardy inequalities for the discrete Dirichlet Laplacian on the half-line is proved. 
\end{abstract}

\maketitle

\section{Introduction}
A huge number of papers have been devoted to effects of 
additive perturbation~$V$ to a~given operator~$H_{0}$ to the spectrum. 
With traditional motivations rooted in mathematical physics, 
$H_{0}$ is a differential operator, 
such as the Schr\"odinger or Dirac operator,
and $V$ stands for a scalar or matrix operator of multiplication.
Typical questions of interests are: 
What is the location of the spectrum of the perturbed operator $H_{0}+V$, 
if $V$ belongs to a given class of potentials? 
How does the best possible spectral enclosure look like?
Are some components of the spectrum preserved
under suitable smallness or repulsive type 
conditions on the potentials?
Et cetera.

If~$H_{0}$ and~$V$ are self-adjoint, 
this type of problems have been intensively studied for more than a century,
due to the needs of quantum mechanics,
and the spectral properties are well understood.
On the other hand, the theory of non-self-adjoint~$V$ (or even~$H_0$)
is much less developed 
and the investigation is essentially restricted to the last two decades.
However, there are new motivations 
(including quantum mechanics \cite{Bagarello-book}),
which make the analysis highly expedient and fashionable.
It is also mathematically challenging because of the lack of tools
based on the spectral theorem.

The new non-self-adjoint era of the aforementioned type of problems
is certainly initiated by the highly influential work
of Abramov, Aslanyan, and Davies from 2001~\cite{abr-asl-dav_jpa01},
where the authors derived an optimal spectral enclosure for 
Schr{\"o}dinger operators on the line with complex-valued potentials. 
Ten years after, the optimal spectral bounds for the case of Schr{\" o}dinger operators on the half-line were deduced by Frank, Laptev, and Seiringer in~\cite{Frank-Laptev-Seiringer_2011}. 
Instead of giving an incomplete list of works with similar goals 
in various settings, we rather refer to recent papers 
\cite{Hansmann-Krejcirik,Mizutani-Schiavone}
with a fairly large collection of references on the subject.
 
The discrete analogue of the celebrated result~\cite{abr-asl-dav_jpa01} 
has been established only recently in~\cite{ibr-sta_ieot19}.
More specifically, the authors derived optimal spectral enclosures
for discrete Schr{\" o}dinger operators on $\Z$ 
with complex-valued potentials in sequence spaces.
The case of discrete Dirac operators on $\Z$ 
was investigated in~\cite{cas-ibr-kre-sta_ahp20}. 
Except for these two works, we are not aware of any paper
on discrete counterparts of the (comparatively many) 
differential settings studied in the last two decades. 

It is precisely the goal of the present paper to continue 
with the research project on spectral properties of
non-self-adjoint discrete operators initiated in 
\cite{ibr-sta_ieot19, cas-ibr-kre-sta_ahp20}.
Here we intend to present a discrete analogue 
of the continuous Schr\"odinger operators on the half-line
studied in~\cite{Frank-Laptev-Seiringer_2011}. 
More specifically, in Theorem~\ref{thm:spectral_enclosure_ell1_pot}, 
we deduce a spectral enclosure for the discrete Schr{\" o}dinger operator on $\N$,
subject to a complex Robin ``boundary condition''
and complex-valued $\ell^{1}$-potentials. 
The optimality is proven in Theorem~\ref{thm:optimality_ell_1_bound}. 
These results are presented in Section~\ref{sec:spec_enclosures}, 
while in Section~\ref{sec:prelim} we introduce the discrete Robin Laplacian on~$\N$
as the unperturbed operator
and analyze its spectral properties as a necessary preliminary.

In the rest of the present paper,
we go beyond the setting of~\cite{Frank-Laptev-Seiringer_2011} 
by looking for conditions on the potentials guaranteeing 
the spectral stability of the discrete Robin Schr{\" o}dinger operators.
Here we adopt the notion of spectral stability 
as used in \cite{Hansmann-Krejcirik}, 
meaning that the point, continuous and residual components 
of the spectra are preserved by the perturbations. 

In this course, discrete Hardy inequalities enter the game. Therefore Section~\ref{sec:hardy} is devoted to Hardy inequalities for the discrete Robin Laplacians on $\N$. As the main result of this section, we prove in Theorem~\ref{thm:generalized_discr_hardy} a general identity that allows to generate optimal discrete Hardy inequalities and, in addition, identifies the remainder term in the inequality. As a~concrete application, we obtain a one parameter family of optimal Hardy weights for the discrete Robin Laplacian on~$\N$ in Theorem~\ref{thm:hardy_ineq_aq}. Although our initial motivation stems from the exploration of the spectral stability, results on the discrete Hardy inequalities of Section~\ref{sec:hardy} are of independent interest.

In Theorems~\ref{thm:cond_empty_point_spec_complex} and~\ref{thm:cond_empty_discr_spec_complex} of Section~\ref{sec:spec_stability}, 
we establish general conditions on the potential guaranteeing the spectral stability of the discrete Robin Schr{\" o}dinger operator on $\N$. Here we restrict the otherwise complex Robin parameter to a real interval, 
making the Dirichlet and Neumann cases as two extreme cases.
(The continuous analogue which is also new is established
as Remark~\ref{Rem.continuous}.)
Finally, we combine these results with the discrete Hardy inequalities 
of Section~\ref{sec:hardy}
in order to deduce more explicit bounds on the potential implying the spectral stability in Theorem~\ref{thm:cond_from_hardy_point_spec}. 

The paper is concluded by Section~\ref{Sec.open},
in which we mention a challenging open problem,
and by Appendix with several illustrative and comparison plots.

\section{The discrete Robin Laplacians on $\N$ and their spectral properties}
\label{sec:prelim}

We start with several definitions to clarify what we mean by discrete analogues of Dirichlet, Neumann, and Robin Laplacians on $\N:=\{1,2,3,\dots\}$. 

First, the discrete differentiation is realized by the following backward and forward difference operators acting on $\ell^{2}(\N)$:
\begin{equation}
 (D\psi)_{n}:=\begin{cases} 
  					\psi_{n-1}\hskip-8pt&-\psi_{n}, \quad n>1,\\
  							  &-\psi_{1}, \quad n=1,
 				 \end{cases}
\quad\mbox{ and }\quad 				 
(D^{*}\psi)_{n}:=\psi_{n+1}-\psi_{n}, \quad n\geq1,
\label{eq:def_diff_op}
\end{equation}
for $\psi\in\ell^{2}(\N)$. 
We could have defined $(D\psi)_{n}:=\psi_{n-1}-\psi_{n}$
for every $n \geq 1$ with the convention that $\psi_0:=0$,
which is the realisation of the ``Dirichlet condition'',
while~$D^*$ satisfies no boundary condition, as in the continuous case.
Then, analogically to the continuous setting, operators $D^{*}D$ and $DD^{*}$ represent the discrete Dirichlet and Neumann Laplacian on $\N$, respectively; see the well-known commutation scenario in the continuous case~\cite[p.~263]{RS4}.

The discrete Dirichlet and Neumann Laplacians are closely related to
Jacobi operators $J_{0}$ and $J_{1}$ given by the tridiagonal matrices
\[
 J_{0}:=\begin{pmatrix}
  0 & 1 & 0 & 0 & \dots \\
  1 & 0 & 1 & 0 & \dots \\
  0 & 1 & 0 & 1 & \dots \\
  0 & 0 & 1 & 0 & \dots \\
  \vdots & \vdots & \vdots & \vdots & \ddots
 \end{pmatrix}
 \quad\mbox{ and }\quad 
 J_{1}:=\begin{pmatrix}
  1 & 1 & 0 & 0 & \dots \\
  1 & 0 & 1 & 0 & \dots \\
  0 & 1 & 0 & 1 & \dots \\
  0 & 0 & 1 & 0 & \dots \\
  \vdots & \vdots & \vdots & \vdots & \ddots
 \end{pmatrix}
\]
with respect to the standard basis 
$\{e_{n}\}_{n\in\N}$ of~$\ell^{2}(\N)$.
Indeed, one has
\[
-\Delta_0 := D^{*}D=2-J_{0} 
\quad\mbox{ and }\quad 
-\Delta_1 := DD^{*}=2-J_{1}.
\]
Since the relationship is through a mere constant shift
and a sign change, spectral properties of the Laplacians 
are encoded in $J_{0}$ and $J_{1}$.
Therefore,
with some abuse of terminology, we will refer to $J_{0}$ and $J_{1}$ as the discrete Dirichlet and Neumann Laplacians on $\N$, too. 

In a~grater generality, we can introduce the one-parameter family of Jacobi operators
\[
J_{a}:=\begin{pmatrix}
  a & 1 & 0 & 0 & \dots \\
  1 & 0 & 1 & 0 & \dots \\
  0 & 1 & 0 & 1 & \dots \\
  0 & 0 & 1 & 0 & \dots \\
  \vdots & \vdots & \vdots & \vdots & \ddots
 \end{pmatrix},
 \qquad 
 a\in\C,
\]
and, in analogy to the above relations, also operators
\begin{equation}
 -\Delta_{a}:=2-J_{a}.
\label{eq:def_delta_a}
\end{equation}
We refer to $-\Delta_{a}$ as well as $J_{a}$ as the discrete Robin Laplacian on $\N$ with the coupling constant $a\in\C$. Clearly, $-\Delta_{0}$ and  $-\Delta_{1}$ are the discrete Dirichlet and Neumann Laplacian on $\N$, respectively. 

While spectral properties of $J_{0}$ are well known, this seems to be not  the case for $J_{a}$ with general $a\in\C$. Therefore we summarize spectral properties of $J_{a}$ in the next theorem. To this end, it is useful to introduce the Joukowski transform
\begin{equation}\label{Joukowski}
 z=z(k):=k+k^{-1},
\end{equation}
where $k\in\C$, $0<|k|<1$. The Joukowski transform $k\mapsto z(k)$ is a one-to-one mapping from the punctured open unit disk $\D\setminus\{0\}$ onto the set $\C\setminus[-2,2]$ (the unit circle $\T$ is mapped twice onto $[-2,2]$). 

\begin{thm}\label{thm:spectrum_J_a}
For $a\in\C$, one has
\[
 \sigma_{\rm{c}}(J_{a})=[-2,2], \quad 
 \sigma_{\rm{r}}(J_{a})=\emptyset, 
 \quad\mbox{ and }\quad 
 \sigma_{\rm{p}}(J_{a})=\begin{cases}
  \emptyset&  \mbox{ if }\, |a|\leq 1,\\
  \{a+a^{-1}\}& \mbox{ if }\, |a|>1.\\
 \end{cases} 
\]
Moreover, if $|a|>1$, the only eigenvalue of $J_{a}$ is simple, i.e., of algebraic multiplicity~$1$.
Further, for $z=k+k^{-1}\notin\sigma(J_{a})$, where $0<|k|<1$, the Green kernel of $J_{a}$ reads
\begin{equation}
 (J_{a}-z)^{-1}_{m,n}=\frac{(k-a)k^{m+n-1}-(k^{-1}-a)k^{|n-m|+1}}{(1-ak)(k-k^{-1})}, \quad m,n\in\N.
\label{eq:green_kernel_Ja}
\end{equation}
\end{thm}

\begin{proof}
 First, we analyze the point spectrum of $J_{a}$. For $k\in\C$, $0<|k|\leq1$, we seek the solution $\psi=\psi(k)$ of the eigenvalue equation 
 \[
 J_{a}\psi=(k+k^{-1})\psi,
 \]
 which is determined uniquely up to a multiplicative constant. One readily verifies that the solution reads
 \begin{equation}
  \psi_n =
\begin{cases}
  \displaystyle
  \psi_{n}(k)=(k-a)k^{n-1}-(k^{-1}-a)k^{-n+1}
  & \mbox{if} \quad k\neq\pm1 \,,
  \\
  u_{n}(\pm1)=(\pm 1)^{n}\left(n\mp a(n-1)\right)
  & \mbox{if} \quad k=\pm1 \,.
\end{cases} 
 \label{eq:psi_eigenvec}
 \end{equation}
 The only non-trivial square summable solution is $\psi=\psi\left(a^{-1}\right)$ provided that $|a|>1$. It follows the assertion about the point spectrum
 $\sigma_{\rm{p}}(J_{a})$.
 
 Second, we show that the resolvent set of $J_{a}$ contains $\C\setminus[-2,2]$ if $|a|\leq1$, and $\C\setminus\left([-2,2]\cup\{a+a^{-1}\}\right)$ if $|a|>1$, i.e., $\C\setminus\left([-2,2]\cup\sigma_{\rm{p}}(J_{a})\right)\subset\rho(J_{a})$. The theory of Jacobi operators provides us with the general formula for the Green kernel
 \[
 G_{m,n}(z) = \frac{1}{W(\psi,\varphi)}\times\begin{cases}
  \varphi_{n}\psi_{m}&\quad\mbox{ if } m\leq n,\\
  \varphi_{m}\psi_{n}&\quad\mbox{ if } n<m,
 \end{cases}
\]
where $\psi$ is a solution of eigenvalue equation $J_{a}\psi=z\psi$ (which need not to belong to~$\ell^{2}(\N)$), $\varphi\in\ell^{2}(\N)$ and fulfills $\varphi_{n-1}+\varphi_{n+1}=z\varphi_{n}$ for all $n>1$, and the $n$-independent Wronskian $W(\psi,\varphi)$ is given by the formula
\[
 W(\psi,\varphi) = \psi_{n+1}\varphi_{n}-\psi_{n}\varphi_{n+1},
\]
which is non-vanishing if and only if $\psi$ and $\varphi$ are linearly independent. In our case, writing~\eqref{Joukowski}, 
the sequence $\psi$ is given by~\eqref{eq:psi_eigenvec} and $\varphi_{n}=k^{n}$, $n\in\N$. Then 
\[
W(\psi,\varphi)=(1-ak)(k-k^{-1})
\]
and 
\[
G_{m,n}(z)=\frac{(k-a)k^{m+n-1}-(k^{-1}-a)k^{n-m+1}}{(1-ak)(k-k^{-1})}, \quad\mbox{ for } m \leq n,
\]
and symmetrically in the case $m>n$. Notice that, for $0<|k|<1$, $W(\psi,\varphi)=0$ if and only if $k=a^{-1}$ and $|a|>1$, which corresponds to the eigenvalue $a+a^{-1}$. Hence $G_{m,n}$ is well defined for all $z\in\C\setminus[-2,2]$ if $|a|\leq 1$, and all $z\in\C\setminus\left([-2,2]\cup\{a+a^{-1}\}\right)$ if $|a|>1$.

Next, we verify that the operator $G(z)$ with matrix elements $G_{m,n}(z)$ is bounded whenever $0<|k|<1$ with the only exception of $k=a^{-1}$ in the case when $|a|>1$. Then, taking also into account that $G(z)(J_{a}-z)=(J_{a}-z)G(z)=I$, as one readily checks, we will know that $\C\setminus\left([-2,2]\cup\sigma_{\rm{p}}(J_{a})\right)\subset\rho(J_{a})$ and also formula~\eqref{eq:green_kernel_Ja} will be proven for all $z\in\C\setminus\left([-2,2]\cup\sigma_{\rm{p}}(J_{a})\right)$.

Notice that
\[
 G(z)=\alpha(k) H(k)+\beta(k)T(k),
\]
where $H(k)$ is the Hankel matrix with entries 
$H_{m,n}(k):=k^{m+n-2}$, $T(k)$ is the Toeplitz matrix with entries 
$T_{m,n}(k):=k^{|m-n|}$, for $m,n\in\N$, and $\alpha(k)$ and $\beta(k)$ are constants. We show that both $H(k)$ and $T(k)$ are bounded for all $k\in\D$ and hence also $G(z)$ is bounded for all $z\in\C\setminus\left([-2,2]\cup\sigma_{\rm{p}}(J_{a})\right)$. The Hankel operator $H(k)$ is actually Hilbert--Schmidt, since 

\[
 \|H(k)\|_{\rm{HS}}=\sqrt{\sum_{m,n=1}^{\infty}\left|H_{m,n}(k)\right|^{2}}=\sum_{n=0}^{\infty}|k|^{2n}=\frac{1}{1-|k|^{2}}<\infty.
\]
Using the general formula for the norm of a Toeplitz operator, see for example~\cite[\S~2.8]{bot-sil_06}, we get
\[
 \|T(k)\|=\max_{\theta\in[-\pi,\pi]}\left|\sum_{n=-\infty}^{\infty}k^{|n|}e^{\ii\theta}\,\right|\leq\sum_{n=-\infty}^{\infty}|k|^{|n|}=\frac{1+|k|}{1-|k|}<\infty.
\]

Third, we show that $[-2,2]\subset\sigma(J_{a})$. This follows from the Weyl theorem since
 \[
  \lim_{n\to\infty}\frac{\|(J_{a}-2\cos\phi)\psi^{(n)}\|}{\|\psi^{(n)}\|}=0,
 \]
 where $\psi^{(n)}:=\left(\psi_{1},\dots,\psi_{n},0,0,\dots\right)$ is truncated sequence~\eqref{eq:psi_eigenvec} with $k=e^{\ii\phi}$ and $\phi\in[-\pi,\pi]$. Indeed, the limit can be readily computed using equations 
 \[
 \|(J_{a}-2\cos\phi)\psi^{(n)}\|^{2}=|\psi_{n+1}|^{2} \quad\mbox{ and }\quad \|\psi^{(n)}\|^{2}=\sum_{j=1}^{n}|\psi_{j}|^{2}
 \]
together with the formulas from~\eqref{eq:psi_eigenvec}.

To conclude the assertion about the spectrum of $J_{a}$ and its parts, 
it is sufficient to verify that 
the residual spectrum is empty.
By definition, $z\in\sigma_{\rm{r}}(J_{a})$ if and only if $z\notin\sigma_{\rm{p}}(J_{a})$ and $\Ran(J_{a}-z)^{\perp}=\Ker(J_{a}-z)^{*}$ is non-trivial. Since $J_{a}^{*}=J_{\overline{a}}$ the letter is equivalent to $z\in\sigma_{\rm{p}}(J_{a})$. Hence $\sigma_{\rm{r}}(J_{a})=\left(\C\setminus\sigma_{\rm{p}}(J_{a})\right)\cap\sigma_{\rm{p}}(J_{a})=\emptyset$.

In summary, the spectrum of~$J_a$ is purely continuous 
and given by the interval $\sigma_\mathrm{c}(J_a)=[-2,2]$, 
except for the existence of the unique discrete eigenvalue 
$\lambda := a+a^{-1}$ if $|a| > 1$.
It remains to verify the simplicity of the eigenvalue. To this end, it suffices to check that the eigenvector $\psi(a^{-1})$ corresponding to the eigenvalue $\lambda$ of $J_{a}$ and the eigenvector $\overline{\psi(a^{-1})}$ corresponding to the eigenvalue $\overline{\lambda}$ of the adjoint $J_{a}^{*}=J_{\overline{a}}$ are not orthogonal. This is straightforward since, by using~\eqref{eq:psi_eigenvec}, we obtain
 \[
  \left\langle \overline{\psi(a^{-1})}, \psi(a^{-1}) \right\rangle=\sum_{n=1}^{\infty}\psi_{n}^{2}(a^{-1})=\left(a^{-1}-a\right)^{2}\,\sum_{n=1}^{\infty}a^{-2n+2}=a^{2}-1\neq0.
 \] 
It completes the proof of the theorem.
\end{proof}

Lastly, we mention a duality between $J_{a}$ and $J_{-a}$ which allows to restrict some parts of the forthcoming analysis to $a\geq0$ without loss of generality. The proof is straightforward.

\begin{prop}\label{prop:unitary}
We have the equations
\[
 UJ_{-a}U^{*}=-J_{-a} \quad\mbox{ and }\quad U(-\Delta_{a})U^{*}=4+\Delta_{-a},
\]
for the unitary operator $U:=\diag(1,-1,1,-1,1,\dots)$ and all $a\in\C$.
\end{prop}

\section{Optimal spectral enclosures for $J_{a}$ perturbed by complex $\ell^{1}$-potentials}
\label{sec:spec_enclosures}

The goal of this section is to deduce optimal spectral bounds for 
the spectrum of the shifted operator $J_{a}+V$, 
where $V:=\diag(v_{1},v_{2},\dots)$ is a diagonal matrix 
with complex sequence $v=\{v_{n}\}_{n=1}^{\infty}\in\ell^{1}(\N)$. We refer to the sequence~$v$ as well as the operator~$V$ as the potential.

For $z=k+k^{-1}$, where $|k|\leq 1$, we define the function
\begin{equation}
 g_{a}(z):=\sup_{n\in\N}\left|1-\frac{k-a}{1-ak}\,k^{2n-1}\right|.
\label{eq:def_g_a}
\end{equation}
If $|a|>1$, then $g_{a}\left(a+a^{-1}\right)=\infty$.

\begin{thm}\label{thm:spectral_enclosure_ell1_pot}
 Let $v\in\ell^{1}(\N_{0})$ and $a\in\C$. Then
 \begin{equation}\label{eq:ell1_bound_optimal}
  \sigma_\mathrm{p}(J_{a}+V)\subset\{z\in\C\setminus(-2,2) 
  \mid \sqrt{|z^{2}-4|}\leq g_{a}(z)\,\|v\|_{\ell^{1}}\}.
 \end{equation}
\end{thm}

\begin{proof}
It follows from the Birman--Schwinger principle, cf.~\cite[Corol.~4]{Hansmann-Krejcirik},
that for any $z\in\C$ (including $z\in \sigma_\mathrm{c}(J_a)=[-2,2]$), 
we have the implication
\[
 \|K(z)\|<1\quad\Rightarrow\quad z\notin\sigma_\mathrm{p}(J_{a}+V),
\]
where $K(z):=|V|^{1/2}(J_{0}-z)^{-1}|V|^{1/2}\sgn V$ is the Birman--Schwinger operator.

With the aid of the Cauchy--Schwarz inequality, we may estimate
 \begin{align*}
  \|K(z)u\|^{2}&\leq\sum_{m=0}^{\infty}\left(\sum_{n=0}^{\infty}\sqrt{|v_{m}|}\left|\left(J_{a}-z\right)_{m,n}^{-1}\right|\sqrt{|v_{n}|}|u_{n}|\right)^{2}\leq\gamma_{a}^{2}(z)\sum_{m=0}^{\infty}|v_{m}|\left(\sum_{n=0}^{\infty}\sqrt{|v_{n}|}|u_{n}|\right)^{\!2}\\
  &\leq\gamma_{a}^{2}(z)\|v\|_{\ell^{1}}^{2}\|u\|^{2},
 \end{align*}
 for any $u\in\ell^{2}(\N)$, where
 \[
  \gamma_{a}(z):=\sup_{m,n\in\N}\left|\left(J_{a}-z\right)^{-1}_{m,n}\right|.
 \]
 Note that the thresholds of the continuous spectrum $\pm2$, 
 as well as the discrete eigenvalue $a+a^{-1}$ if $|a|>1$, 
 are always included in the set on the right-hand side of~\eqref{eq:ell1_bound_optimal}. Next, using~\eqref{eq:green_kernel_Ja} for $z=k+k^{-1}$ with $|k|\leq1$, where $k\neq\pm1$ and $k\neq a^{-1}$ if $|a|>1$, we obtain
 \begin{align*}
  \gamma_{a}(z)&=\sup_{\substack{n,m\in\N \\ m\leq n}}|k|^{n-m}\left|\frac{1-ak-(k-a)k^{2m-1}}{(1-ak)(k^{-1}-k)}\right|=\frac{1}{|k^{-1}-k|}\,\sup_{m\in\N}\left|1-\frac{k-a}{1-ak}\,k^{2m-1}\right|\\
  &=\frac{g_{a}(z)}{\sqrt{z^{2}-4}}.
 \end{align*}
 In total, we see that
 \[ 
 \|K(z)\|\leq\gamma_{a}(z)\|v\|_{\ell^{1}}=\frac{g_{a}(z)}{\sqrt{z^{2}-4}}\|v\|_{\ell^{1}},
 \]
 for $z\neq\pm2$, and hence
 \[
  \sigma_\mathrm{p}(J_{a}+V)\subset\{z\in\C \mid \sqrt{|z^{2}-4|}\leq g_{a}(z)\,\|v\|_{\ell^{1}}\}.
 \]

Moreover, we can remove the interval $(-2,2)$ from the spectral enclosure since it follows from the assumption $v\in\ell^{1}(\N)$ that $\sigma_\mathrm{p}(J_{a}+V)\cap(-2,2)=\emptyset$. This is a well known result based on an asymptotic behavior of the so-called Jost solutions of $J_{0}+V$, see for example \cite[Sec.~13.6]{Simon_OPUC2}.
More specifically, the reference deals with
real $v\in\ell^{1}(\N)$ only, 
however, the reality of $v$ is not essential for the proof and the same assertion holds true for complex $v\in\ell^{1}(\N)$, too (we are not aware of an exact reference for this results covering complex potentials).
\end{proof}

To prove the optimality of~\eqref{eq:ell1_bound_optimal} we will need the following auxiliary result.

\begin{lem}\label{lem:optimality_aux}
For any $\kappa\in\C$ and $k\in\C\setminus\R$, there exists $n\in\N$ such that $|1+\kappa k^{2n}|\geq 1$.
\end{lem}

\begin{proof}
 Let $\kappa\in\C$ and $k\in\C\setminus\R$ be fixed. The claim holds trivially for $\kappa=0$, therefore we further assume $\kappa\neq0$. Then $|1+\kappa k^{2n}|\geq 1$ if $\Re(\kappa k^{2n})\geq0$ for some $n\in\N$. This is true if there exists $n\in\N$ such that
 \[
  \Re e^{\ii\phi+2n\ii\theta}\geq0,
 \]
 where $\phi:=\arg\kappa$ and $\theta:=\arg k$. 
 
If $\theta\in\pi(\R\setminus\Q)$, then the set $\{e^{\ii\phi+2n\ii\theta} \mid n\in\N\}$ is dense in the unit circle and hence the claim is obviously true. Next, suppose $\theta\in\pi\Q$. Notice that $\theta\notin\pi\Z$ since $k\notin\R$ by assumptions. This means that the points from $\{e^{\ii\phi+2n\ii\theta} \mid n\in\N\}$ are vertices of a regular $p$-gon with $p\geq2$. At least one of these vertices has to be located in the half-plane $\Re z\geq0$.
\end{proof}

\begin{thm}\label{thm:optimality_ell_1_bound}
The spectral bound~\eqref{eq:ell1_bound_optimal} is optimal in the following sense: given $Q>0$ and $z\in\C\setminus\R$ such that $\sqrt{|z^{2}-4|}=g_{a}(z)Q$ there exist $n\in\N$ and $\omega\in\C$ with $|\omega|=Q$ such that $z$ is the only eigenvalue of $J_{a}+\omega P_{n}$, where $P_{n}=\langle\,\cdot\,,e_{n}\rangle e_{n}$.
\end{thm}

\begin{rem}
 Note that the optimality result does not apply to real boundary points of the spectral enclosure~\eqref{eq:ell1_bound_optimal}. It is clear that the boundary points belonging to the open interval $(-2,2)$ cannot be eigenvalues. However, it remains an open problem whether the boundary points located in $\R\setminus[-2,2]$ can be eigenvalues of $J_{a}+V$ for a potential $V$.
\end{rem}

\begin{proof}[Proof of Theorem~\ref{thm:optimality_ell_1_bound}]
Let $Q>0$ and $z\in\C\setminus\R$ such that $\sqrt{|z^{2}-4|}=g_{a}(z)Q$ be given. 
We consider the potential of the form $V=\omega P_{n}$ with $|\omega|=Q$ and determine $n\in\N$ as well as the argument of $\omega$ 
such that $z\in\sigma_\mathrm{p}(J_{a}+\omega P_{n})$. By~\eqref{eq:green_kernel_Ja}, 
the Birmann--Schwinger operator $K(z)$ corresponding to $V=\omega P_{n}$ is
\[
 K(z)=\frac{\omega}{k^{-1}-k}\left(1-\frac{k-a}{1-ak}k^{2n-1}\right)P_{n}.
\]
It follows that $z=k+k^{-1}$, where $k\in\D\setminus(-1,1)$ and $k\neq a^{-1}$, if $|a|>1$, is an eigenvalue of $J_{a}+\omega P_{n}$, if and only if 
\begin{equation}\label{eq:char_eq_inproof}
\frac{\omega}{k^{-1}-k}\left(1-\frac{k-a}{1-ak}k^{2n-1}\right)=-1.
\end{equation}

Note that the supremum in definition~\eqref{eq:def_g_a} of $g_{a}(z)$ is attained for some $n\in\N$ which follows from Lemma~\ref{lem:optimality_aux} and the fact that
\[
 \lim_{n\to\infty}\left(1-\frac{k-a}{1-ak}k^{2n-1}\right)=1.
\]
Thus, we pick $n\in\N$ such that 
\[
 g_{a}(z)=\left|1-\frac{k-a}{1-ak}k^{2n-1}\right|.
\]
It implies that moduli of both sides in~\eqref{eq:char_eq_inproof} coincide. Now, it suffices to set
\[
 \arg\omega :=\pi-\arg\left[\frac{1}{k^{-1}-k}\left(1-\frac{k-a}{1-ak}k^{2n-1}\right)\right]
\]
and equation~\eqref{eq:char_eq_inproof} is fulfilled.
\end{proof}

By estimating $g_{a}(z)$ from above in~\eqref{eq:ell1_bound_optimal}, 
one can obtain various non-optimal spectral enclosures whose advantage may be their simpler form. For example, one has
\[
 g_{a}(z)\leq 1+\left|\frac{k-a}{1-ak}\right|,
\]
which implies the following corollary. 

\begin{cor}
For $v\in\ell^{1}(\N_{0})$ and $a\in\C$, we have
 \begin{equation}\label{eq:ell1_bound_a}
  \sigma_\mathrm{p}(J_{a}+V)
  \subset\{\pm2\}\cup\left\{k+k^{-1} \;\big|\; 0<|k|<1, \; \left|k^{-1}-k\right|\left|1-ak\right|\leq|\left(|1-ak|+|k-a|\right)\|v\|_{\ell^{1}}\right\}\!.
 \end{equation}
\end{cor}

Recall that in the case of discrete Schr{\" o}dinger operators on $\Z$ with complex $\ell^{1}$-potentials, see~\cite[Thm.~1.1]{ibr-sta_ieot19}, the optimal spectral enclosure looks the same as~\eqref{eq:ell1_bound_optimal} but the function $g_{a}$ is not present. This is quite analogous to the known results 
in the continuous setting:
While the optimal spectral enclosure of Schr{\" o}dinger operators on the line with integrable complex potentials is a disk centered at the origin, see~\cite[Thm.~4]{abr-asl-dav_jpa01}, it is deformed in the case of Schr{\" o}dinger operators on the half-line with Dirichlet boundary condition by an influence of an extra term similar to the function $g_{a}$, see \cite[Thm.~1.1]{Frank-Laptev-Seiringer_2011}, and also \cite{Enblom_2017}.

Due to the presence of~$g_{a}$ in~\eqref{eq:ell1_bound_a}, the geometry of the spectral enclosure is highly non-trivial. Figures~\ref{fig:spec_bound_dirichlet} and~\ref{fig:spec_bound_neumann} show the boundary curves of the spectral enclosures for discrete Dirichlet and Neumann Schr{\" o}dinger operators given by the equations
\[
 \sqrt{|z^{2}-4|}=g_{a}(z)Q, \quad a\in\{0,1\},
\]
for several values of the parameter $Q=\|v\|_{\ell^{1}}$, where
\[
 g_{0}\left(k^{-1}+k\right)=\sup_{n\in\N}\left|1-k^{2n}\right| \quad\mbox{ and }\quad 
 g_{1}\left(k^{-1}+k\right)=\sup_{n\in\N}\left|1+k^{2n-1}\right|.
\]
Three more illustrative plots for other values of $a$ are postponed to Appendix.

\begin{figure}[H]
    \centering
	\includegraphics[width=0.99\textwidth]{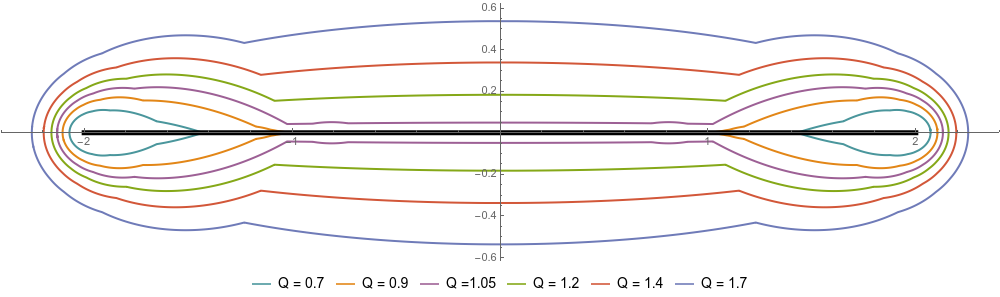}
    \caption{Spectral enclosures for the discrete Dirichlet Schr{\" o}dinger operator ($a=0$).}
    \label{fig:spec_bound_dirichlet}
\end{figure}
\begin{figure}[H]
	\includegraphics[width=0.99\textwidth]{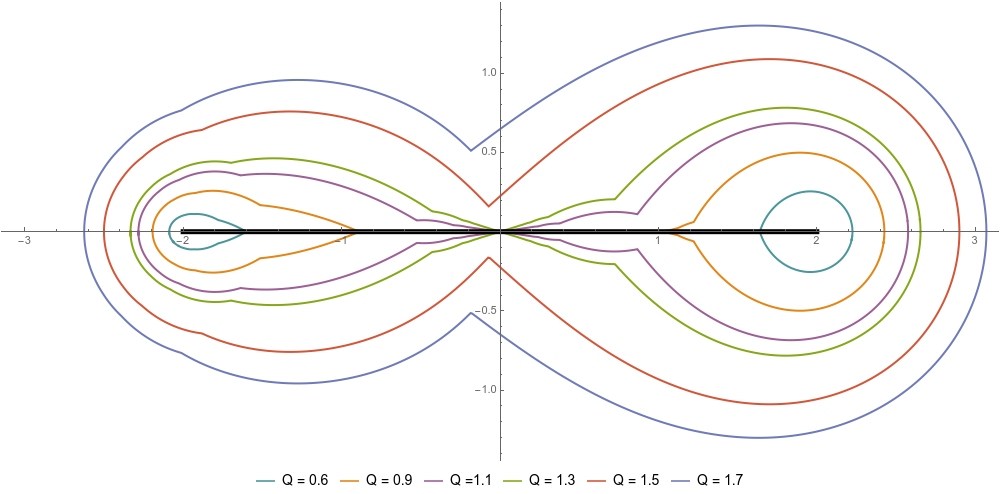}
	\caption{Spectral enclosures for the discrete Neumann Schr{\" o}dinger operator ($a=1$).}
	\label{fig:spec_bound_neumann}
\end{figure}

\section{Intermezzo: Optimal discrete Hardy inequalities for $-\Delta_{a}$}
\label{sec:hardy}

We interrupt the analysis of spectral properties of complex perturbations of discrete Laplacians by a study of discrete Hardy inequalities. These inequalities will be applied in the forthcoming section in results on a spectral stability of $-\Delta_{a}+V$ for small complex potentials $V$. Nevertheless, results of this section should be of independent interest.

Recall that, for $u\in\ell^{2}(\N)$, the classical discrete Hardy inequality reads
\begin{equation}\label{eq:hardy_classical}
 \sum_{n=1}^{\infty}\left|u_{n-1}-u_{n}\right|^{2}\geq\sum_{n=1}^{\infty}\frac{|u_{n}|^{2}}{4n^{2}},
\end{equation}
where $u_{0}:=0$; see~\cite{kuf-mal-per_amm06} for a historical account of the inequality. If we temporarily denote by $W$ the diagonal operator defined by equations $We_{n}:=4n^{-2}e_{n}$, for $n\in\N$, and recall definition~\eqref{eq:def_delta_a} with $a=0$, we may rewrite~\eqref{eq:hardy_classical} as the operator inequality 
\begin{equation}
  -\Delta_{0}\geq W,
  \label{eq:hardy_operator_form}
\end{equation}
understood in the sense of quadratic forms. 
More generally,
if~\eqref{eq:hardy_operator_form} holds for 
some
$0\leq W=\diag(w_{1},w_{2},\dots)$, 
the operator~$-\Delta_0$ is said to satisfy a Hardy inequality and
the sequence $w=\{w_{n}\}_{n=1}^{\infty}$ is said to be a~Hardy weight.

A surprising fact observed by Keller, Pinchover, and Pogorzelski~\cite{kel-pin-pog_amm18, kel-pin-pog_cmp18} is that the classical discrete Hardy inequality~\eqref{eq:hardy_classical} is not optimal in the following sense: A Hardy weight $w$ is said to be an \emph{optimal} Hardy weight, if for any other Hardy weight $\tilde{w}$, the point-wise inequality $\tilde{w}\geq w$ implies $\tilde{w}=w$. In~\cite{kel-pin-pog_cmp18}, it is proved that the optimal Hardy weight actually reads
\begin{equation}
 w_{n}:=2-\sqrt{1-\frac{1}{n}}-\sqrt{1+\frac{1}{n}}>\frac{1}{4n^{2}}, \quad n\in\N.
\label{eq:optimal_hardy_weight}
\end{equation}
At this point, the discrete and continuous case are not completely analogical since, in the continuous case, the well-known Hardy inequality
\[
 \int_{0}^{\infty}|u'(x)|^{2}\dd x\geq\int_{0}^{\infty}\frac{|u(x)|^{2}}{4x^{2}}\dd x
\]\
valid for every $u\in W^{1,2}_{0}((0,\infty))$
is optimal; see~\cite[Sec.~8.1]{wei_cpde99} and references therein.

\begin{rem}\label{rem:stronger_optimality}
In fact, a stronger notion of optimality, which involves two additional conditions on $w$, was considered in~\cite{kel-pin-pog_cmp18},
based on the pioneering works 
\cite{Devyver-Fraas-Pinchover_2012, Devyver-Fraas-Pinchover_2014}
in the continuous case. 
Namely, a~Hardy weight $w$ is said to be optimal, if and only if the following three conditions hold:
\begin{enumerate}
\item[(opt1)] If $\tilde{w}$ is a Hardy weight such that $\tilde{w}\geq w$, then $\tilde{w}=w$,
\item[(opt2)] $\ker(-\Delta_{0}-W)=\{0\}$,
\item[(opt3)] $(\forall \epsilon>0)(\forall n\in\N)(\exists\psi\in\ell^{2}(\N))(\|DQ_{n}\psi\|^{2}<(1+\epsilon)\langle\psi,WQ_{n}\psi\rangle)$,
\end{enumerate}
where $W=\diag(w_{1},w_{2},\dots)$, $D$ is the difference operator defined in~\eqref{eq:def_diff_op}, and $Q_{n}$ is the orthogonal projection onto $\spn\{e_{n},e_{n+1},\dots\}$. The Hardy weight~\eqref{eq:optimal_hardy_weight} enjoys all the three properties (opt1--3).
\end{rem}
 
Our next goal is to establish the discrete Hardy inequalities for the discrete Robin Laplacian, i.e., to replace $-\Delta_{0}$ in~\eqref{eq:hardy_operator_form} by $-\Delta_{a}$. At this point, we restrict the parameter $a$ to the interval $[0,1]$, so the Dirichlet and the Neumann cases correspond 
to the two extreme points~$0$ and~$1$, respectively. 
This restriction is without loss of generality,
because the concept of Hardy inequalities is meaningful
for the self-adjoint setting $a \in \R$ only,
the couplings $a \not\in [-1,1]$ are disregarded because of the existence 
of a discrete eigenvalue (see Theorem~\ref{thm:spectrum_J_a})
and the case of $a \in [-1,0]$ can be deduced from $a \in [0,1]$
by the duality of Proposition~\ref{prop:unitary}.

In fact, it makes sense to consider only $a\in[0,1)$ since, in the Neumann case, there is no Hardy inequality analogously to the well-known continuous case. In other words, the operator $-\Delta_{1}$ is critical in the sense of the following statement.

\begin{thm}
 If $W\geq0$ is a diagonal operator such that $-\Delta_{1}\geq W$, then $W=0$.
\end{thm}

\begin{proof}
Suppose that the inequality $-\Delta_{1}\geq W$ holds for some $W\geq0$. We show that it implies $W=0$. The idea is based on the observation that the constant sequence $\psi\equiv 1$ is anihilated by $-\Delta_{1}$. Since $1\notin\ell^{2}(\N)$ we need to use a regularized sequence which can be chosen, for instance, as follows:
\[
 \psi_{n}^{(N)}:=\begin{cases}
  1, \quad &1\leq n<N,\\
  \frac{2N-n}{N}, \quad &N\leq n\leq 2N,\\ 
  0, \quad &2N<n,
 \end{cases}
\]
for $n,N\in\N$. Notice that $\psi^{(N)}\to 1$ point-wise as $N\to\infty$ and $\psi^{(N)}\leq\psi^{(N+1)}$ for all $N\in\N$. Further, an easy calculation shows that
\[
 \left\langle\psi^{(N)},-\Delta_{1}\psi^{(N)}\right\rangle=\left\|D^{*}\psi^{(N)}\right\|^{2}=\sum_{n=1}^{\infty}\left|\psi^{(N)}_{n+1}-\psi^{(N)}_{n}\right|^{2}=\frac{1}{N}\to0 \quad \mbox{ as } N\to\infty,
\]
while
\[
 \left\langle\psi^{(N)},W\psi^{(N)}\right\rangle=\sum_{n=1}^{\infty}w_{n}\left|\psi^{(N)}_{n}\right|^{2}\to \sum_{n=1}^{\infty}w_{n} \quad \mbox{ as } N\to\infty,
\]
by the Monotone Convergence Theorem. It follows from the assumption $-\Delta_{1}\geq W$ that
\[
\left\langle\psi^{(N)},-\Delta_{1}\psi^{(N)}\right\rangle\geq \left\langle\psi^{(N)},W\psi^{(N)}\right\rangle, \quad \forall N\in\N.
\]
By sending $N\to\infty$, we obtain inequality $\sum_{n=1}^{\infty}w_{n}\leq 0$. Since $w_{n}\geq0$ for all $n\in\N$, we conclude that $w_{n}=0$ for all $n\in\N$.
\end{proof}

Since the discrete Robin Laplacian $-\Delta_{a}$ is a rank-one perturbation of $-\Delta_{0}$, namely
\begin{equation}
 -\Delta_{a}=-\Delta_{0}-aP_{1},
\label{eq:delta_a_rank-one_perturb}
\end{equation}
where $P_{1}:=\langle\,\cdot\,,e_{1}\rangle e_{1}$
as in Theorem~\ref{thm:optimality_ell_1_bound}, 
we can deduce Hardy inequalities for $-\Delta_{a}$ from those for $-\Delta_{0}$.

Next, we prove an abstract identity which yields a method for generating Hardy inequalities for the discrete Dirichlet Laplacian. Moreover, it determines the reminder term in the Hardy inequlity and, via a control of the reminder, provides a~sufficient condition guaranteeing the optimality. Recall that here we mean by the optimality the property (opt1) from Remark~\ref{rem:stronger_optimality}. Spaces of finitely supported sequences indexed by $\N$ and $\N_{0}:=\N \cup \{0\}$ are denoted by $C_{0}(\N)$ and $C_{0}(\N_{0})$, respectively.

\begin{thm}\label{thm:generalized_discr_hardy}
 Let $g=\{g_{n}\}_{n=1}^{\infty}$ be a positive sequence such that $-\Delta_{0}g\geq0$ entrywise. 
 Then, for any $u\in C_{0}(\N_{0})$ with $u_{0}=0$, we have the identity
 \begin{equation}
  \sum_{n=1}^{\infty}\left|u_{n}-u_{n-1}\right|^{2}=\sum_{n=1}^{\infty}w_{n}|u_{n}|^{2}+\sum_{n=2}^{\infty}\left|\sqrt{\frac{g_{n-1}}{g_{n}}}u_{n}-\sqrt{\frac{g_{n}}{g_{n-1}}}u_{n-1}\right|^{2},
 \label{eq:gen_discr_hardy_id}
 \end{equation}
 where
 \[
  w_{n}:=\frac{(-\Delta_{0}g)_{n}}{g_{n}}.
 \]
 In particular, it follows the generalized discrete Hardy inequality
 \[
  \sum_{n=1}^{\infty}w_{n}|u_{n}|^{2}\leq\sum_{n=1}^{\infty}\left|u_{n}-u_{n-1}\right|^{2}.
 \] 
 Moreover, if there exists a sequence of elements $\xi^{N}\in C_{0}(\N)$ such that $\xi^{N}\leq\xi^{N+1}$, $\xi^{N}\to 1$ as $N\to\infty$ pointwise, and
 \begin{equation}
  \lim_{N\to\infty}\,\sum_{n=2}^{\infty}g_{n}g_{n-1}\left|\xi_{n}^{N}-\xi_{n-1}^{N}\right|^{2}=0,
 \label{eq:opt_cond_gen_discr_hardy}
 \end{equation}
 then $w$ is optimal.
\end{thm}

\begin{proof}
 Let us temporarily denote $h:=-Dg$. Then for any $u\in C_{0}(\N_{0})$ with $u_{0}=0$, one has
 \begin{align*}
  \left|\sqrt{1-\frac{h_{n}}{g_{n}}}u_{n}-\sqrt{1+\frac{h_{n}}{g_{n-1}}}u_{n-1}\right|^{2}\!=
  \left(1-\frac{h_{n}}{g_{n}}\right)|u_{n}|^{2}+\!\left(1+\frac{h_{n}}{g_{n-1}}\right)|u_{n-1}|^{2}-2\Re\left(\overline{u}_{n}u_{n-1}\right)\!,
 \end{align*}
 which further implies the identity
 \[
 \left|\sqrt{\frac{g_{n-1}}{g_{n}}}u_{n}-\sqrt{\frac{g_{n}}{g_{n-1}}}u_{n-1}\right|^{2}=\left|u_{n}-u_{n-1}\right|^{2}-h_{n}\left(\frac{|u_{n}|^{2}}{g_{n}}-\frac{|u_{n-1}|^{2}}{g_{n-1}}\right),
 \]
 for all $n\in\N$, where the terms 
 \[
 \sqrt{\frac{g_{n}}{g_{n-1}}}u_{n-1} \quad\mbox{ and }\quad \frac{|u_{n-1}|^{2}}{g_{n-1}}
 \]
 are to be interpreted as zeros for $n=1$. Then by summing by parts, we obtain
 \begin{align*}
  \sum_{n=1}^{\infty}w_{n}|u_{n}|^{2}&=\sum_{n=1}^{\infty}\left(h_{n}-h_{n+1}\right)\frac{|u_{n}|^{2}}{g_{n}}=\sum_{n=1}^{\infty}h_{n}\left(\frac{|u_{n}|^{2}}{g_{n}}-\frac{|u_{n-1}|^{2}}{g_{n-1}}\right)\\
  &=\sum_{n=1}^{\infty}\left|u_{n}-u_{n-1}\right|^{2}-\sum_{n=1}^{\infty}\left|\sqrt{\frac{g_{n-1}}{g_{n}}}u_{n}-\sqrt{\frac{g_{n}}{g_{n-1}}}u_{n-1}\right|^{2},
 \end{align*}
 which establishes~\eqref{eq:gen_discr_hardy_id} since the first term of the last sum vanishes.
 
 Next, let a sequence of elements $\xi^{N}\in C_{0}(\N)$ satisfying the assumptions is given. Suppose further that $\tilde{w}$ is another Hardy weight, i.e.,
 \[
 \sum_{n=1}^{\infty}\tilde{w}_{n}|u_{n}|^{2}\leq\sum_{n=1}^{\infty}\left|u_{n}-u_{n-1}\right|^{2},
 \]
 for all $u\in C_{0}(\N_{0})$ with $u_{0}=0$, and $\tilde{w}\geq w$. Then, by~\eqref{eq:gen_discr_hardy_id}, we have
 \[
  0\leq\sum_{n=1}^{\infty}\left(\tilde{w}_{n}-w_{n}\right)|u_{n}|^{2}\leq \sum_{n=2}^{\infty}\left|\sqrt{\frac{g_{n-1}}{g_{n}}}u_{n}-\sqrt{\frac{g_{n}}{g_{n-1}}}u_{n-1}\right|^{2}
 \]
 for all $u\in C_{0}(\N_{0})$ with $u_{0}=0$. Plugging $u_{n}:=g_{n}\xi_{n}^{N}$ into the last expression, we get
 \[
   0\leq\sum_{n=1}^{\infty}\left(\tilde{w}_{n}-w_{n}\right)\left|g_{n}\xi_{n}^{N}\right|^{2}\leq\sum_{n=2}^{\infty}g_{n}g_{n-1}\left|\xi_{n}^{N}-\xi_{n-1}^{N}\right|^{2}.
 \]
 Using~\eqref{eq:opt_cond_gen_discr_hardy} and the assumptions $\xi^{N}\leq\xi^{N+1}\to 1$, we arrive, by the Monotone Convergence Theorem, at the equality
 \[
 \sum_{n=1}^{\infty}\left(\tilde{w}_{n}-w_{n}\right)g_{n}^{2}=0.
 \]
 Since each term of the sum is non-negative and $g_{n}\neq0$ for all $n\in\N$, we conclude that $w=\tilde{w}$.
\end{proof}

\begin{rem}
 Theorem~\ref{thm:generalized_discr_hardy} is a generalization of \cite[Thm.~1]{kre-sta_amm21}.
\end{rem}

An interesting corollary of Theorem~\ref{thm:generalized_discr_hardy} is that we have an infinitely many fully explicit optimal discrete Hardy inequalities.

\begin{cor}\label{cor:opt_hardy_n^q}
 For any $q\in(0,1/2]$, we have the optimal discrete Hardy inequality
 \[
  \sum_{n=1}^{\infty}w_{n}(q)|u_{n}|^{2}\leq\sum_{n=1}^{\infty}|u_{n}-u_{n-1}|^{2},
 \]
 where 
 \[
  w_{n}(q):=2-\left(1-\frac{1}{n}\right)^{q}-\left(1+\frac{1}{n}\right)^{q}=
  \begin{cases}
  \displaystyle 2-2^{q}&\quad\mbox{ if } n=1,\\[3pt]
  \displaystyle 2q\sum_{k=1}^{\infty}\frac{(1-q)_{2k-1}}{(2k)!}\frac{1}{n^{2k}}&\quad\mbox{ if } n\geq 2.
  \end{cases}
 \]
\end{cor}

\begin{rem}
 Note that, for all $n\geq2$ and $q\in(0,1/2)$, one has
 \[
  w_{n}(q)<w_{n}\!\left(\frac{1}{2}\right),
  \qquad \mbox{while} \qquad
  w_{1}(q)>w_{1}\!\left(\frac{1}{2}\right).
 \]
 If $q\in(1/2,1)$, the discrete Hardy inequality with the weight $w_{n}(q)$ still holds but it is not optimal since $w_{n}(q)<w_{n}(1/2)$ for all $n\in\N$. 
 For $q>1$, the weight $w_{n}(q)$ is not non-negative, and $w(1)$ is trivial.
\end{rem}

\begin{proof}[Proof of Corollary~\ref{cor:opt_hardy_n^q}]
 To obtain the inequality, it suffices to put $g_{n}:=n^{q}$ in Theorem~\ref{thm:generalized_discr_hardy}.  
 
 To prove the optimality, we define
 \[
    \xi_n^N :=
  \begin{cases}
    1 & \mbox{if} \quad n < N \,,
    \\[2pt]
    \frac{\displaystyle 2 \log{N}-\log n}{\displaystyle\log N}
    & \mbox{if} \quad N\leq n \leq N^2 \,,
    \\[2pt]
    0 & \mbox{if} \quad n > N^2 \,,
  \end{cases}
 \]
 for $N\geq2$. Then $\xi^{N}\in C_{0}(\N)$ such that $\xi^{N}\leq\xi^{N+1}$, $\xi^{N}\to 1$ as $N\to\infty$. Hence it remains to check condition~\eqref{eq:opt_cond_gen_discr_hardy}. 
We have
\begin{align*}
\sum_{n=2}^\infty n^{q}(n-1)^{q}\left|\xi_{n}^{N}-\xi_{n-1}^{N}\right|^{2}
&=\frac{1}{\log^{2}N}\sum_{n=N+1}^{N^{2}}n^{q}(n-1)^{q}\log^{2}\left(\frac{n}{n-1}\right)\\
&\leq\frac{1}{\log^{2}N}\sum_{n=N+1}^{N^{2}}\frac{n^{q}(n-1)^{q}}{(n-1)^{2}}
\leq\frac{2}{\log^{2}N}\sum_{n=N+1}^{N^{2}}\frac{1}{(n-1)^{2-2q}}\\
&\leq\frac{2}{\log^{2}N}\sum_{n=N+1}^{N^{2}}\frac{1}{(n-1)}\leq\frac{2}{\log^{2}N}\int_{N}^{N^{2}}\frac{\dd n}{n-1}\\
&=\frac{2\log(N+1)}{\log^{2}N}\leq\frac{4}{\log N}.
\end{align*}
Since the last expression tends to zero as $N\to\infty$, we are done.
\end{proof}

\begin{rem}
Clearly, if $q=1/2$, Corollary~\ref{cor:opt_hardy_n^q} reestablishes the result of~\cite{kel-pin-pog_amm18} (see also \cite{kre-sta_amm21}). 
Concerning the stronger notion of optimality from Remark~\ref{rem:stronger_optimality}, for $q\in(0,1/2)$, 
one only has (opt1) and (opt2) but not (opt3).
\end{rem}

Now we may combine Corollary~\ref{cor:opt_hardy_n^q} with~\eqref{eq:delta_a_rank-one_perturb} to deduce optimal discrete Hardy inequalities for the Robin Laplacian $-\Delta_{a}$. However, depending on the value of $a\in[0,1)$, the range for the parameter $q\in[0,1/2)$ has to be additionally restricted in order to the corresponding Hardy weight remains non-negative. To this end, for $a\in[0,1)$, we denote
\begin{equation}
 q_{a}:=\min\left(\log_{2}(2-a),\,\frac{1}{2}\right).
\label{eq:def_q_a}
\end{equation}

\begin{thm}[discrete Hardy inequalities for $-\Delta_{a}$]\label{thm:hardy_ineq_aq}
For any $a\in[0,1)$ and $q\in(0,q_{a}]$, the Hardy inequality
\[
 -\Delta_{a}\geq W
\]
holds, where $W=\diag(w_{1},w_{2},\dots)$ with 
\[
 w_{n}=2-\left(1-\frac{1}{n}\right)^{q}-\left(1+\frac{1}{n}\right)^{q}-a\delta_{n,1}.
\]
Moreover, all these Hardy weights are optimal.
\end{thm}

\section{The spectral stability of $-\Delta_{a}$ perturbed by small complex potentials}
\label{sec:spec_stability}

In this section, we provide sufficient conditions on the complex potential $V$ that guarantee emptiness of either the point or the discrete spectrum of $J_{a}+V$ for $a\in[0,1)$.

As an auxiliary result, we need to maximize the function
\begin{equation}
 g_{m}(\theta;a):=\frac{\left|\sin(m\theta)-a\sin((m-1)\theta)\right|}{|\sin(\theta)|\,\sqrt{1+a^{2}-2a\cos\theta}},
\label{eq:def_g_m}
\end{equation}
for $\theta\in[-\pi,\pi]$, where $a\in[0,1)$ and $m\in\N$ are parameters.
The value $g_m(0,a)$ is conventionally defined
as the continuous extension of $\theta \mapsto g_m(\theta,a)$ 
initially defined on $[-\pi,0) \cup (0,\pi]$.

\begin{lem}\label{lem:maximum_of_g_m}
For any $a\in[0,1)$ and $m\in\N$, one has
\[
 \max_{\theta\in[-\pi,\pi]}g_{m}(\theta;a)=g_{m}(0;a)=\frac{m-a(m-1)}{1-a}.
\]
\end{lem}

\begin{proof}
Notice the function $g_{m}(\,\cdot\,;a)$ is even, hence the analysis can be restricted to $[0,\pi]$.
We will need the following two inequalities. 

First, for all $\theta \in (0,\pi)$ and $m\in\N$, we have
\begin{equation}
  \frac{|\sin(m\theta)|}{\sin(\theta)}\leq m.
  \label{eq:auxineq1_inproof}
\end{equation}
Inequality~\eqref{eq:auxineq1_inproof} can be proven by induction in $m\in\N$. The induction step makes use of the inequality
\begin{align*}
 \frac{|\sin((m+1)\theta)|}{\sin(\theta)} &= \frac{|\sin(m\theta)\cos(\theta)+\cos(m\theta)\sin(\theta)|}{\sin(\theta)}
  \leq \frac{|\sin(m\theta)|}{\sin(\theta)} |\cos(\theta)|
  +|\cos(m\theta)|\\
  &\leq \frac{|\sin(m\theta)|}{\sin(\theta)} + 1,
\end{align*}
which holds true for all $\theta\in(0,\pi)$.

Second, for all $\theta\in(0,\pi)$, $a\in[0,1)$ and $m\in\N$, we have
\begin{equation}
  f_m(\theta;a):=
  \frac{m(1-a\cos\theta)+a}{m(1-a)+a}
  \leq \frac{1-a\cos\theta}{1-a}
  = \lim_{m\to\infty} f_m(\theta;a).
  \label{eq:auxineq2_inproof}
\end{equation}
This is true since, for $\theta\in(0,\pi)$ and $a\in[0,1)$ fixed, $f_m(\theta;a)$ is increasing in $m>0$. Indeed, the differentiation yields
\[
  \frac{\partial f_m(\theta;a)}{\partial m} = \frac{a^{2}(1-\cos\theta)}{[m(1-a)+a]^2} > 0.
\]

Now, we can prove the statement of the lemma. With the aid of~\eqref{eq:auxineq1_inproof}, for $\theta\in(0,\pi)$ and $a\in[0,1)$, we obtain 
\[
   g_{m}(\theta;a)
   =\frac{\left|\sin(m\theta) \, [1-a\cos(\theta)]+a \sin(\theta) \cos(m\theta)\right|}{\sin(\theta)\sqrt{1+a^{2}-2a\cos(\theta)}}
   \leq \frac{m \, [1-a\cos(\theta)]+ a }{\sqrt{1+a^{2}-2a\cos(\theta)}}.
\]
It follows that the statement of the lemma holds provided that 
\[
  \frac{m \, [1-a\cos(\theta)] + a }{m (1- a) + a} 
  \leq \frac{\sqrt{1+a^{2}-2a\cos(\theta)}}{1-a},
\]
which, in turn, is true if 
\[
  1-a\cos(\theta)  \leq \sqrt{1+a^{2}-2a\cos(\theta)} 
\]
by~\eqref{eq:auxineq2_inproof}.
By squaring, the last inequality is equivalent to $a^{2} \cos^2(\theta) \leq a^2$, which is always true. 
\end{proof} 

Now, we are in a position to prove the theorem on a spectral stability of the discrete Robin Schr{\" o}dinger operator on $\N$. Recall that, by Theorem~\ref{thm:spectrum_J_a}, $\sigma(J_{a})=[-2,2]$ for $a\in[0,1)$.

\begin{thm}\label{thm:cond_empty_point_spec_complex}
Let $a\in[0,1)$, $v=\{v_{n}\}_{n=1}^{\infty}$ be a complex sequence, and $V=\diag(v_{1},v_{2},\dots)$. If the semi-infinite matrix $K'$ with elements
\[
 K'_{m,n}:=\sqrt{|v_{m}|}\left[\frac{a}{1-a}+\min(m,n)\right] \sqrt{|v_{n}|}, \quad m,n\in\N,
\] 
regarded as an operator on~$\ell^{2}(\N)$, satisfies $\|K'\|<1$, then $\sigma(J_{a}+V)=\sigma_\mathrm{c}(J_{a}+V)=\sigma(J_{a})$. 
Equivalently, if there exists a constant $c<1$ such that 
\begin{equation}
 |V|\leq c\left(-\Delta_{a}\right),
\label{eq:cond_func_form}
\end{equation}
we have $\sigma(J_{a}+V)=\sigma_\mathrm{c}(J_{a}+V)=\sigma(J_{a})$.
\end{thm}

\begin{proof}
Let $a\in[0,1)$. For all $m,n\in\N$, we prove that
\begin{equation}
 \sup_{z\in\C}\left|(J_{a}-z)_{m,n}^{-1}\right|=\left|(J_{a}-2)_{m,n}^{-1}\right|=\frac{a}{1-a}+\min(m,n).
\label{eq:green_func_J_a_bound}
\end{equation}
The verification of~\eqref{eq:green_func_J_a_bound} is postponed to the end of the proof.
Having~\eqref{eq:green_func_J_a_bound}, we may proceed similarly as in the proof of Theorem~\ref{thm:spectral_enclosure_ell1_pot} and estimate the norm of the Birman--Schwinger operator this time as follows:
\begin{align*}
 \left|\left\langle\phi,K(z)\psi\right\rangle\right|&\leq\sum_{m,n=1}^{\infty}|\phi_{n}|\sqrt{|v_{m}|}\left|\left(J_{a}-z\right)_{m,n}^{-1}\right|\sqrt{|v_{n}|}|\psi_{n}|\\
 &\leq\sum_{m,n=1}^{\infty}|\phi_{n}|\sqrt{|v_{m}|}\left[\frac{a}{1-a}+\min(m,n)\right]\sqrt{|v_{n}|}|\psi_{n}|\\
 &=(|\phi|,K'|\psi|)
\end{align*}
for any $\phi,\psi\in\ell^{2}(\N)$ and $z\in\C$. It follows that $\|K(z)\|\leq\|K'\|$ for all $z\in\C$. Hence, if $\|K'\|<1$, the spectral stability follows by \cite[Thm.~3]{Hansmann-Krejcirik}.

Next, we prove the equivalence between the inequality $\|K'\|<1$ and~\eqref{eq:cond_func_form}. The operator inequality~\eqref{eq:cond_func_form} can be written as
\[
 \langle\psi,|V|\psi\rangle\leq c\langle\psi,(2-J_{a})\psi\rangle,
\]
for all $\psi\in\ell^{2}(\N)$, which holds true if and only if
\[
 \left\||V|^{1/2}\psi\right\|^{2}\leq c  \left\|\left(2-J_{a}\right)^{1/2}\psi\right\|^{2},
\]
for all $\psi\in\ell^{2}(\N)$. Yet another equivalent form of the inequality reads
\[
\left\||V|^{1/2}(2-J_{a})^{-1/2}\phi\right\|^{2}\leq c \, \|\phi\|^{2},
\]
for all $\phi\in\Ran(2-J_{a})^{1/2}$. Since $2\in\sigma_{c}(J_{a})$, see Theorem~\ref{thm:spectrum_J_a}, the range of $(2-J_{a})^{1/2}$ is dense in $\ell^{2}(\N_{0})$. This means that $|V|^{1/2}(2-J_{a})^{-1/2}$ extends to a bounded operator with the norm 
\[
\left\||V|^{1/2}(2-J_{a})^{-1/2}\right\|\leq\sqrt{c}.
\]
Finally, it suffices to note that
\[
 \|K'\|=\left\||V|^{1/2}(2-J_{a})^{-1}|V|^{1/2}\right\|=\left\||V|^{1/2}(2-J_{a})^{-1/2}\right\|^{2}\leq c<1,
\]
where we have used the fact that $\|TT^{*}\|=\|T\|^{2}$ for a bounded operator $T$.

It remains to verify~\eqref{eq:green_func_J_a_bound}. By inspection of formula~\eqref{eq:green_kernel_Ja}, one observes that the Green kernel $(J_{a}-k-k^{-1})_{m,n}^{-1}$ is an analytic function in $k$ in the open unit disk and continuous to its boundary. Indeed, the fact that the Green kernel extends continuously also to the points $k=\pm1$ follows from~\eqref{eq:green_kernel_Ja} and limit relations
\[
\frac{(k-a)k^{m+n}-(1-ak)k^{n-m+1}}{(1-ak)(1-k^{2})}\to
\begin{cases}
-\frac{m-a(m-1)}{1-a}& \quad \mbox{ as }\,k\to1,\\[2pt]
(-1)^{m+n}\,\frac{m+a(m-1)}{1+a}& \quad \mbox{ as }\,k\to-1,
\end{cases}
\]
that can be verified by a straightforward computation. Thus, by the Maximum Modulus Principle, one has
\[
 \sup_{z\in\C}\left|(J_{a}-z)_{m,n}^{-1}\right|=\max_{|k|\leq1}\left|\left(J_{a}-k-k^{-1}\right)_{m,n}^{-1}\right|=\max_{|k|=1}\left|\left(J_{a}-k-k^{-1}\right)_{m,n}^{-1}\right|.
\]
By setting $k=e^{\ii\theta}$, for $\theta\in[-\pi,\pi]$, and using~\eqref{eq:green_kernel_Ja}, we get
\[
\left|\left(J_{a}-2\cos\phi\right)_{m,n}^{-1}\right|=\left|\frac{\left(e^{\ii\theta}-a\right)e^{\ii(m+n-1)\theta}-\left(e^{-\ii\theta}-a\right)e^{\ii(-m+n+1)\theta}}{\left(1-ae^{\ii\theta}\right)\left(e^{\ii\theta}-e^{-\ii\theta}\right)}\right|=g_{m}(\theta;a),
\]
if $m\leq n$, where the function $g_{m}(\,\cdot\,;a)$ is defined in~\eqref{eq:def_g_m}. By Lemma~\ref{lem:maximum_of_g_m}, we conclude that
\[
\max_{|k|=1}\left|\left(J_{a}-k-k^{-1}\right)_{m,n}^{-1}\right|=
g_{m}(0;a)=\frac{a}{1-a}+m,
\]
if $1\leq m\leq n$. Equation~\eqref{eq:green_func_J_a_bound} now follows from the symmetry of the Green kernel in~$m$ and~$n$.
\end{proof}

\begin{rem} 
The stability of the spectrum of $J_a + V$
under the condition $\|K'\| < 1$ goes back to 
an old smoothness idea of Kato's \cite{Kato_1966}. 
In fact, it follows by his \cite[Thm.~1.5]{Kato_1966}
that $J_a + V$ is similar to $J_a$ under the condition $\|K'\| < 1$.
The stability of the spectrum of $J_a + V$
under the subordination condition~\eqref{eq:cond_func_form}   
goes back to the idea of~\cite[Thm.~1]{FKV} in a continuous setting.
The equivalence between the two conditions 
was realized in \cite{Hansmann-Krejcirik},
to where we refer for a survey, further abstract developments 
and applications.
\end{rem}

\begin{rem}\label{rem:stability_negative_a}
The claim of Theorem~\ref{thm:cond_empty_point_spec_complex} can be readily generalized to $a\in(-1,1)$. In fact, the condition $\|K'\|<1$ implies $\sigma(J_{a}+V)=\sigma_\mathrm{c}(J_{a}+V)=\sigma(J_{a})$ also for $a\in(-1,0)$ because $\sigma(J_{-a}+V)=-\sigma(J_{a}-V)$ and similarly for each part of the spectra, which follows from Proposition~\ref{prop:unitary}, and the operator $K'$ is defined in terms of the absolute value $|V|$. 
On the other hand, condition~\eqref{eq:cond_func_form} has to be replaced by $c(4+\Delta_{a})\geq|V|$ if $a\in(-1,0)$, which is again a~consequence of Proposition~\ref{prop:unitary}.
\end{rem}

\begin{cor}\label{Corol1}
Suppose $a\in(-1,1)$. If the potential $V=\diag(v_{1},v_{2},\dots)$ satisfies
\begin{equation}
 \sum_{m,n=1}^{\infty}|v_{m}|\left[\frac{a}{1-a}+\min(m,n)\right]^{2}|v_{n}|<1,
\label{eq:cond_empty_point_spec_complex}
\end{equation}
or even the stricter condition 
\begin{equation}\label{weighted}
 \sum_{n=1}^{\infty} \left(\frac{a}{1-a}+n^2\right)|v_{n}|<1 \,,
\end{equation}
then $\sigma(J_{a}+V)=\sigma_\mathrm{c}(J_{a}+V)=\sigma(J_{a})$. 
\end{cor}
\begin{proof}
First, note that the left-hand side of~\eqref{eq:cond_empty_point_spec_complex} coincides with the square of the Hilbert--Schmidt norm of the operator $K'$ from Theorem~\ref{thm:cond_empty_point_spec_complex}. Then~\eqref{eq:cond_empty_point_spec_complex} implies 
$\|K'\|\leq\|K'\|_\mathrm{HS}<1$ and hence the claim follows from Theorem~\ref{thm:cond_empty_point_spec_complex} and Remark~\ref{rem:stability_negative_a}. Second, it suffices to note that \eqref{weighted} implies \eqref{eq:cond_empty_point_spec_complex} which is a consequence of the inequality $\left(\alpha+\min(m,n)\right)^{2}\leq (\alpha+m^2)(\alpha+n^2)$ that holds true for all $m,n\in\N$ and $\alpha\geq0$.
\end{proof}

Note that the subordination of the sufficient conditions
of Theorem~\ref{thm:cond_empty_point_spec_complex}
and Corollary~\ref{Corol1}
is as follows:
$
  \eqref{eq:cond_func_form} 
  \, \Leftarrow\,
  \eqref{eq:cond_empty_point_spec_complex}
  \, \Leftarrow\,
  \eqref{weighted}
$.

\begin{rem}\label{Rem.continuous}
It is interesting to compare Theorem~\ref{thm:cond_empty_point_spec_complex}
with its continuous analogue.
Since the latter is not available in 
\cite{Frank-Laptev-Seiringer_2011} (nor~\cite{Enblom_2017}),
we establish the result here.
Alternative conditions established by a completely different technique
(including higher dimensions)
can be found in~\cite{CK2}.

Let~$H_\alpha$ denote the Laplacian in $L^2((0,\infty))$,
subject to the Robin boundary condition $\psi'(0) = \alpha \psi(0)$ with $\alpha \in \C$.
By convention, the case $\alpha=+\infty$ is included as the Dirichlet Laplacian.
More specifically, 
$
  \Dom(H_\alpha) := \{\psi \in W^{2,2}((0,\infty))\,|\, \psi'(0) = \alpha \psi(0)\}
$
if $\alpha \in \C$ and 
$
  \Dom(H_\alpha) := \{\psi \in W^{2,2}((0,\infty))\,|\, \psi(0) = 0\}
$
if $\alpha = +\infty$.
For every $z \in \C \setminus [0,+\infty)$,
the resolvent $(H_\alpha-z)^{-1}$ is the integral operator with kernel
$$
  G_z(x,x') := \frac{e^{-\sqrt{-z}\,|x-x'|} - e^{-\sqrt{-z}\,|x+x'|}}
  {2\sqrt{-z}} + \frac{e^{-\sqrt{-z}\,|x+x'|}}{\sqrt{-z}+\alpha}
  \,,
$$
where we consider the principal branch of the square root.

In analogy with Theorem~\ref{thm:cond_empty_point_spec_complex},
let us now restrict to real $\alpha \in (0,+\infty]$.
Then the spectrum of~$H_\alpha$ equals $[0,+\infty)$
and it is purely continuous. 
Let $V \in L^1_\mathrm{loc}((0,\infty))$ 
be relatively form bounded with respect to~$H_\alpha$
with the relative bound less than one.  
Define $H_V := H_\alpha \dot{+} V$, where the sum on the righ-hand side
should be interpreted in the sense of forms.
 
It is not difficult to see that the pointwise inequality
$$
  |G_z(x,x')| \leq |G_0(x,x')| 
  = \frac{1}{\alpha} + \frac{\big||x-x'|-|x+x'|\big|}{2} 
  = \frac{1}{\alpha} + \min(x,x')
$$
holds true for every $z \in \C \setminus [0,+\infty)$
and $x,x' \in (0,\infty)$.
Applying the Birman--Schwinger principle
\cite[Thm.~3]{Hansmann-Krejcirik},
the spectral stability 
$$
  \sigma(H_V) = \sigma_\mathrm{c}(H_V) = \sigma(H_0)
$$
holds true 
(in particular, the point spectrum is empty)
whenever the integral operator~$K'$ with kernel
$$
  |V(x)|^{1/2} \, \left(\frac{1}{\alpha} + \min(x,x')\right) \, |V(x')|^{1/2}
$$
has norm strictly less than one, 
or equivalently, there exists a constant $c<1$ such that
\begin{equation}\label{continuous1}
  \int_0^\infty V(x) \, |u(x)|^2 \, \diff x
  \leq c \left(
  \int_0^\infty |u'(x)|^2 \, \diff x
  + \alpha \, |u(0)|^2
  \right)
\end{equation}
for every $u \in W^{1,2}((0,\infty))$ if $\alpha>0$
or $u \in W_0^{1,2}((0,\infty))$ if $\alpha=+\infty$
(the term $\alpha \, |u(0)|^2$ is interpreted as zero in the latter case).
Note that this subordination condition particularly ensures
that~$V$ is relatively form bounded with respect to~$H_\alpha$
with the relative bound less than one.
Estimating the operator norm of~$K'$
by the Hilbert--Schmidt norm,
a sufficient condition reads
\begin{equation}\label{continuous2}
  \iint\displaylimits_{(0,\infty)^2} 
  |V(x)| \ \left(\frac{1}{\alpha} + \min(x,x') \right)^2 \ |V(x')|
  \ \diff x \, \diff x'
  < 1
  \,.
\end{equation}
Since $\min(x,x')^2 \leq (1+x^2)(1+x'^2)$, for $x,x'\geq0$, an obvious sufficient condition to guarantee~\eqref{continuous2} reads
\begin{equation}\label{continuous3}
  \int_0^\infty |V(x)| \, \left[1+\left(\frac{1}{\alpha}+x\right)^{\!2}\right] \diff x < 1 
  \,.
\end{equation}
Obviously, the sufficient conditions 
\eqref{continuous1}, \eqref{continuous2} and~\eqref{continuous3}
are continuous analogues of 
\eqref{eq:cond_func_form}, 
\eqref{eq:cond_empty_point_spec_complex}
and \eqref{weighted}, respectively.
\end{rem}

Although the weaker assumption $\|K'\|\leq1$, where $K'$ is as in Theorem~\ref{thm:cond_empty_point_spec_complex}, 
does not guarantee $\sigma_\mathrm{p}(J_{a}+V)=\emptyset$, for $a\in[0,1)$, it follows at least that the discrete spectrum
$\sigma_\mathrm{d}(J_{a}+V)$ is empty. In other words, we still have $\sigma(J_{a}+V)=[-2,2]$ but the spectrum of $J_{a}+V$ need not be purely continuous, i.e., the existence of eigenvalues embedded in the interval $[-2,2]$ cannot be excluded.

\begin{thm}\label{thm:cond_empty_discr_spec_complex}
 Suppose that $a\in[0,1)$ and $\{v_{n}\}_{n=1}^{\infty}\subset\C$ is such that $\|K'\|\leq1$, where $K'$ is as in Theorem~\ref{thm:cond_empty_point_spec_complex}, then $\sigma_\mathrm{d}(J_{a}+V)=\emptyset$. Equivalently, if
 \begin{equation}
 |V|\leq-\Delta_{a},
\label{eq:cond_func_form_weaker}
\end{equation}
then $\sigma_\mathrm{d}(J_{a}+V)=\emptyset$.
\end{thm}

\begin{proof}
 Suppose $V=\diag(v_{1},v_{2},\dots)$ is such that $\|K'\|\leq1$. For $q\in(0,1)$, we define an auxiliary operator $K'_{q}$ corresponding to the potential $qV$, i.e., $K'_{q}=qK'$. Since 
 \[ 
 \|K'_{q}\|=q\|K'\|\leq q<1,
 \] 
 we have $\sigma_\mathrm{d}(J_{a}+qV)=\emptyset$ for all $q\in(0,1)$ by Theorem~\ref{thm:cond_empty_point_spec_complex}. 

 Clearly, $J_{a}+qV\to J_{a}+V$ uniformly as $q\to1-$. Consequently, if there exists $\lambda\in\sigma_\mathrm{d}(J_{a}+V)$, then there must be a discrete eigenvalue of $J_{a}+qV$ in a neighborhood of $\lambda$ for~$q$ sufficiently close to $1$ contradicting $\sigma_\mathrm{d}(J_{a}+qV)=\emptyset$. Therefore $\sigma_\mathrm{d}(J_{a}+V)=\emptyset$.

 The proof of the equivalence between the condition $\|K'\|\leq1$ and~\eqref{eq:cond_func_form_weaker} follows the same lines as in the proof of Theorem~\ref{thm:cond_empty_point_spec_complex}.
\end{proof}

\begin{rem}
Analogically as in Remark~\ref{rem:stability_negative_a}, the first statement of Theorem~\ref{thm:cond_empty_discr_spec_complex} holds true with no change even for $a\in(-1,0)$, while condition~\eqref{eq:cond_func_form_weaker} is to be replaced by the inequality $4+\Delta_{a}\geq|V|$ if $a\in(-1,0)$.
\end{rem}

The following statement can be deduced from Theorem~\ref{thm:cond_empty_discr_spec_complex} similarly as Corollary~\ref{Corol1} from Theorem~\ref{thm:cond_empty_point_spec_complex}.

\begin{cor}
If $a\in(-1,1)$ and potential $V=\diag(v_{1},v_{2},\dots)$ fulfills
\begin{equation}
 \sum_{m,n=1}^{\infty}|v_{m}|\left[\frac{a}{1-a}+\min(m,n)\right]^{2}|v_{n}|\leq1,
\label{eq:cond_double_sum_empty_discr_spec}
\end{equation}
or even
\[
 \sum_{n=0}^{\infty}\left(\frac{a}{1-a}+n^2\right)|v_{n}|\leq1,
\]
then $\sigma_\mathrm{d}(J_{a}+V)=\emptyset$.
\end{cor}

Finally, we may deduce more concrete conditions on the potential $V$ guaranteeing the spectral stability of $J_{a}+V$ by combining Theorems~\ref{thm:cond_empty_point_spec_complex} and~\ref{thm:cond_empty_discr_spec_complex} with the Hardy weights given in Theorem~\ref{thm:hardy_ineq_aq}.

\begin{thm}\label{thm:cond_from_hardy_point_spec}
 Let $a\in[0,1)$. 
 If the complex potential $V=\diag(v_{1},v_{2},\dots)$ satisfies
 \begin{equation}
   |v_{n}|\leq c\left[2-\left(1-\frac{1}{n}\right)^{q}-\left(1+\frac{1}{n}\right)^{q}-a\delta_{n,1}\right],\quad   \forall n\in\N,
   \label{eq:cond_from_hardy_point_spec}
 \end{equation}
 for a constant $c<1$ and $q\in(0,q_{a}]$, where $q_{a}$ is defined by~\eqref{eq:def_q_a}, then 
 \[ 
 \sigma(J_{a}+V)=\sigma_\mathrm{c}(J_{a}+V)=[-2,2].
 \]
 If condition~\eqref{eq:cond_from_hardy_point_spec} is fulfilled with $c=1$, then $\sigma_\mathrm{d}(J_{a}+V)=\emptyset$.
 \end{thm}

\section{An open problem}\label{Sec.open}

As a final remark, we emphasize an interesting research problem.
It is related to the possibility of the extension
of the present method for discrete Schr\"odinger operators
on the half-line which are made critical by subtracting 
the optimal Hardy weight. 

In the continuous setting, the Schr{\" odinger} operator with the optimal Hardy potential
\[
  H:=-\frac{\dd^{2}}{\dd x^{2}}-\frac{1}{4x^{2}}
\]
acting in $L^{2}((0,\infty))$
and subject to Dirichlet boundary condition at $x=0$
is critical, which means that by adding 
an arbitrary negative potential to~$H$
makes the operator not positive any more.
Since the eigenvalue problem for~$H$ 
turns out to be related to
a particular Bessel differential equation,
spectral properties of $H$ can be deduced in terms of well known special functions, see for example~\cite{der-rich_ahp17}. Then one may consider $H$ as an unperturbed operator and investigate, for example, spectral enclosures for perturbations of $H$ by small complex potentials, asymptotic analysis of discrete eigenvalues of the perturbed operator under various settings, etc.

A discrete variant of the continuous problem above
aims to spectral properties of the Jacobi operator $J_{0}-W$ 
with the perturbing potential 
determined by the optimal discrete Hardy weight
\eqref{eq:optimal_hardy_weight}.
However, solutions of 
the difference equation of the eigenvalue problem for 
$J_{0}-W$ 
does not seem to be
expressible in terms of known special functions. 
Consequently, no explicit formula for the Green kernel of $J_{0}-W$ 
seems to be available. A more detailed spectral analysis of the critical operator $J_{0}-W$ would be of great interest. 
In adddition to the perturbation analysis mentioned above,
another interest comes form theory of orthogonal polynomials.
Indeed, the corresponding family of orthogonal polynomials 
determined by the recurrence 
\[
 p_{n+1}(x)=\left(x-\sqrt{1-\frac{1}{n}}-\sqrt{1+\frac{1}{n}}\right)p_{n}(x)-p_{n-1}(x), \quad n\in\N,
\]
with $p_{0}(x)=1$ and $p_{1}(x)=x-\sqrt{2}$, seems not to be analyzed yet.

\section*{Acknowledgement}
The research of D.K. and F.{\v S}. was partially supported 
by the EXPRO grant No.~20-17749X
of the Czech Science Foundation (GA\v{C}R).

\section*{Appendix: Illustrative and comparison plots}

Below in Figures~\ref{fig:spec_bound_a12}, \ref{fig:spec_bound_a2} and~\ref{fig:spec_bound_golden}, we provide illustrative plots of the optimal spectral enclosures of Theorem~\ref{thm:spectral_enclosure_ell1_pot} for $a\in\{1/2, 2, \ii(1+\sqrt{5})/2\}$. Namely, the plots show the boundary curves given by the equation 
\[
 \sqrt{|z^{2}-4|}=g_{a}(z)Q,
\]
for several values of the parameter $Q=\|v\|_{\ell^{1}}$. For $a=2$ and $a=\ii(1+\sqrt{5})/2$, the spectrum of the unperturbed operator~$J_{a}$ contains the extra eigenvalue $a+a^{-1}$, see Theorem~\ref{thm:spectrum_J_a}, that is designated by a red dot.

\begin{figure}[H]
    \centering
	\includegraphics[width=0.99\textwidth]{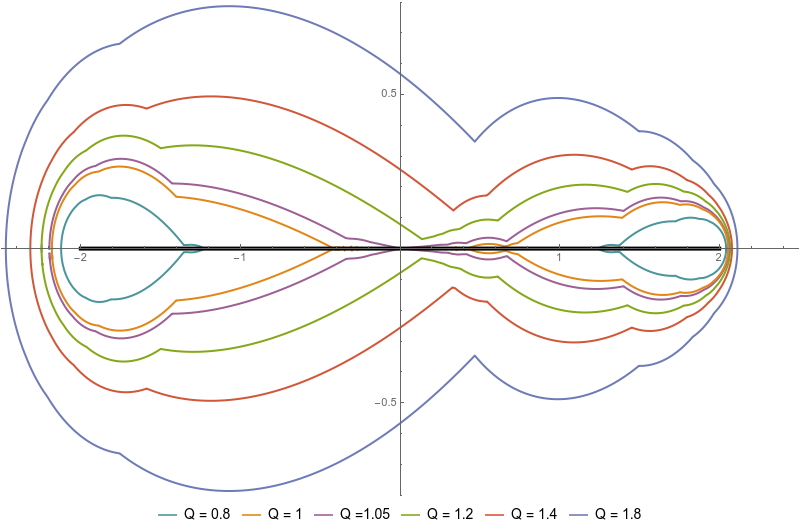}
    \caption{Optimal spectral enclosures~\eqref{eq:ell1_bound_optimal} for $a=1/2$ and several values of $Q=\|v\|_{\ell^{1}}$.}
    \label{fig:spec_bound_a12}
\end{figure}
\begin{figure}[H]
	\includegraphics[width=0.99\textwidth]{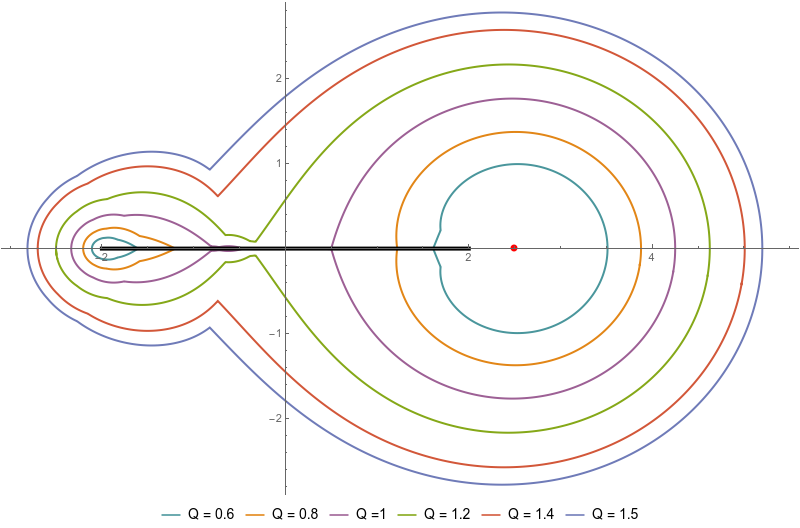}
	\caption{Optimal spectral enclosures~\eqref{eq:ell1_bound_optimal}  for $a=2$ and several values of $Q=\|v\|_{\ell^{1}}$. The red dot designates the sole eigenvalue of $J_{a}$.}
	\label{fig:spec_bound_a2}
\end{figure}
\begin{figure}[H]
	\includegraphics[width=0.99\textwidth]{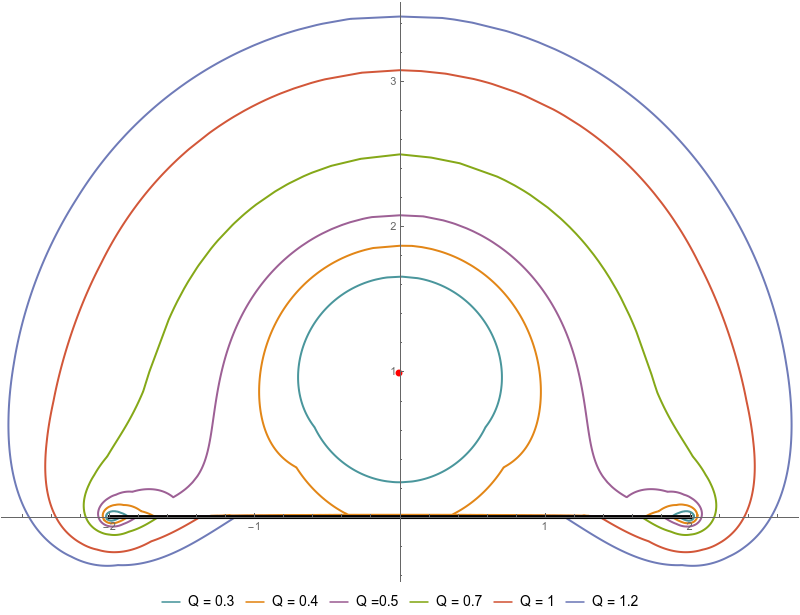}
	\caption{Optimal spectral enclosures~\eqref{eq:ell1_bound_optimal}  for $a=\ii(1+\sqrt{5})/2$ and several values of $Q=\|v\|_{\ell^{1}}$. The red dot designates the sole eigenvalue of $J_{a}$.}
	\label{fig:spec_bound_golden}
\end{figure}

\bibliographystyle{acm}
\bibliography{ref23}

\end{document}